\documentclass[11pt]{amsart}
\usepackage{fullpage}   

\usepackage{amsmath}
\usepackage{amsfonts}
\usepackage{amssymb}
\usepackage{amsthm}
\usepackage{subfigure} % outdated should use subfig or (sub)caption
\usepackage{graphicx}
\usepackage{color}
\usepackage{enumerate}
\usepackage{lineno}
\usepackage{pinlabel}
\usepackage{hyperref}
\usepackage{comment}

\newtheorem{theorem}{Theorem}[section]
\newtheorem{proposition}[theorem]{Proposition}
\newtheorem{lemma}[theorem]{Lemma}

\newtheorem{corollary}[theorem]{Corollary}

\theoremstyle{definition}
\newtheorem{definition}[theorem]{Definition}

\theoremstyle{remark}

\newcommand{\ba}{\setminus}
\newcommand{\F}{\mathcal{F}}
\newcommand{\btu}{\bigtriangleup}
\newcommand{\BL}{\overline{L}}
\newcommand{\BA}{\overline{A}}
\newcommand{\al}{\alpha}
\newcommand{\be}{\beta}

\begin{document}

%\linenumbers

\title{The structure of delta-matroids with width one twists}

\author[C.C.]{Carolyn Chun}
\address{United States Naval Academy, Annapolis, MD, 21402 USA.}
\email{chun@usna.edu}

\author[R.H.]{Rhiannon Hall}
\address{Department of Mathematics, Brunel University London, Uxbridge, Middlesex, UB8 3PH, United Kingdom.}
\email{rhiannon.hall@brunel.ac.uk}

\author[C.M.]{Criel Merino}
\address{Instituto de Matem\'aticas, 
Universidad nacional Aut\'onoma de M\'exico, 
Ciudad de M\'exico, 04510
M\'exico. Investigaci\'on realizada gracias al Programa UNAM-DGAPA-PAPIIT IN102315}
\email{merino@matem.unam.mx}

\author[I.~Moffatt]{Iain Moffatt}
\address{Department of Mathematics,
Royal Holloway,
University of London,
Egham,
Surrey,
TW20 0EX,
United Kingdom.}
\email{iain.moffatt@rhul.ac.uk}

\author[S.N.]{Steven Noble}
\address{Department of Economics, Mathematics and Statistics, Birkbeck, University of London, Malet Street, London, WC1E 7HX, United Kingdom.}
\email{s.noble@bbk.ac.uk}

\subjclass[2010]{05B35}
\keywords{delta-matroid, matroid,  partial dual, excluded minor, twist, width}
\date{\today}

\begin{abstract}  
The width of a delta-matroid is the difference in size between a  maximal and minimal feasible set. 
We give a Rough Structure Theorem for delta-matroids that admit a twist of width one. We 
apply this theorem to give an excluded minor characterisation of delta-matroids that admit a twist of width at most one. 
\end{abstract}

\maketitle

\section{Introduction, results and notation}\label{s1}

Delta-matroids are a generalisation of matroids introduced by A.~Bouchet in \cite{ab1}. They can be thought of as generalising topological graph theory in the same way that matroids can be thought of as generalising graph theory (see, e.g.,  \cite{CMNR1}). Roughly speaking, delta-matroids arise by dropping the requirement that bases are of the same size in the standard definition of a matroid in terms of its bases. (Formal definitions are provided below.) In the context of delta-matroids these generalised ``bases'' are called ``feasible sets''. A basic parameter of a delta-matroid is its ``width'', which is the difference between the sizes of a largest and a smallest of its feasible sets. One of the most fundamental operations in delta-matroid theory is the ``twist''. In this paper we examine how the structure of a delta-matroid determines the width of the delta-matroids that are in its equivalence class under twists.

  Formally, a \emph{delta-matroid} $D=(E,\F)$ consists of a finite set $E$ and a non-empty set $\mathcal{F}$ of subsets of $E$ that  satisfies the \emph{Symmetric Exchange Axiom}:
for all $X,Y\in \mathcal{F}$, if there is an element $u\in X\triangle Y$, then there is an element $v\in X\triangle Y$ such that $X\triangle \{u,v\}\in \mathcal{F}$. 
Here $X\triangle Y$ denotes the symmetric difference of sets $X$ and $Y$. Note that it may be the case that $u=v$ in the Symmetric Exchange Axiom.
Elements of $\mathcal{F}$ are called \emph{feasible sets} and $E$ is the \emph{ground set}. We often use $\mathcal{F}(D)$ and $E(D)$ to denote the set of feasible sets and the ground set, respectively, of $D$. A \emph{matroid} is a delta-matroid whose feasible sets are all of the same size. In this case the feasible sets are called  \emph{bases}. This definition of a matroid is a straightforward reformulation of the standard one in terms of bases.

In general a delta-matroid has feasible sets of different sizes. The \emph{width} of a delta-matroid, denoted $w(D)$, is the difference between the sizes of its largest and smallest feasible sets:
$ w(D):=  \max\limits_{F\in \F}   |F| - \min\limits_{F\in \F}   |F|$.

Twists, introduced by Bouchet in~\cite{ab1}, are one of the fundamental operations of delta-matroid theory. Given a delta-matroid $D=(E,{\mathcal{F}})$ and some subset $A\subseteq E$, the  \emph{twist} of $D$ with respect to $A$, denoted by $D* A$, is the delta-matroid given by $(E,\{A\btu F :  F\in \mathcal{F}\})$.   (At times we write $D\ast e$ for $D\ast \{e\}$.) Note that the ``empty twist'' is $D*\emptyset = D $. The \emph{dual} of $D$, written $D^*$, is equal to $D*E$. Moreover, in general, the twist can be thought of as a ``partial dual'' operation on delta-matroids.

Forming the twist of a delta-matroid usually changes the sizes of its feasible sets and its width. Here we are interested in the problem of recognising when a delta-matroid has a twist of small width. Our results are a Rough Structure Theorem for delta-matroids that have a twist of width one, and an excluded minor characterisation of delta-matroids that have a twist of width at most one. 

To state the Rough Structure Theorem we need the following. Let $D=(E,\mathcal{F})$ be a delta-matroid and let $\mathcal{F}_{\min}$ be the set of feasible sets of minimum size. Then  $D_{\min}:=(E,\mathcal{F}_{\min})$ is a matroid.  
For  a matroid $M$ with ground set $E$, a subset $A$ of $E$ is said to be a \emph{separator} of $M$ if $A$ is a union of components of $M$. Note that both $\emptyset$ and $E$ are always separators. In terms of the matroid rank function, where the rank $r(X)$ of a set $X\subseteq E$ is defined to be the size of the largest intersection of $X$ with a basis of $M$, the set $A$ is a separator if and only if $r(A) + r(E-A) = r(M)$. Throughout the paper we use $\BA$ for the complement $E- A$ of $A$, and $D|X$ denotes the restriction of $D$ to $X\subseteq E$ (see the beginning of Section~\ref{s2} for its definition).

We now state the first of our two main results: a Rough Structure Theorem for delta-matroids admitting a twist of width one.

\begin{theorem}\label{tm1}
Let $D=(E,\F)$ be a delta-matroid. Then $D$ has a twist of width one if and only if there is some $A\subseteq E$ such that 
\begin{enumerate}
\item $A$ is a separator of $D_{\min}$,
\item $D|A$ is a matroid, and
\item $D|\BA$ is of width one.
\end{enumerate}
\end{theorem}
We  actually prove a result that is stronger than Theorem~\ref{tm1}. This stronger result appears below as Theorem~\ref{tt} and  the present theorem follows immediately from it.

As an application of Theorem~\ref{tm1}, we find an excluded minor characterisation of the class of delta-matroids that have a twist of width one as our second main result, Theorem~\ref{t1}.    This class of delta-matroids is shown to be minor closed in   Proposition~\ref{p1}, and its set of excluded minors comprises the delta-matroids in the following definition together with their twists. 

\begin{definition}\label{d1}
Let $D_1$ denote the delta-matroid on the elements $a,b$ with feasible sets
\[\F(D_1)=\{ \emptyset, \{a\}, \{b\}, \{a,b\}   \}.\]
For $i=2,\ldots, 5$ let  $D_i$ denote the delta-matroid on the elements $a,b,c$ with feasible sets given by
\begin{align*}
\F(D_2)&=\{ \emptyset, \{a\}, \{b\}, \{c\},\{a,b,c\}   \}, \\
\F(D_3)&=\{ \emptyset, \{a,b\}, \{b,c\}, \{a,c\}   \}, \\
\F(D_4)&=\{ \emptyset, \{a,b\}, \{b,c\}, \{a,c\},\{a,b,c\}   \}, \\
\F(D_5)&=\{ \emptyset, \{a\}, \{a,b\}, \{b,c\}, \{a,c\}  \}.
\end{align*}
Throughout this paper $D_1, \ldots,D_5$  refer exclusively to these delta-matroids.
Let $\mathcal{D}_{[5]}$ be the set of all twists of these delta-matroids.  Note that $D_i\in \mathcal{D}_{[5]}$ for all $i\in\{1,2,\dots ,5\}$ via the empty twist.
\end{definition}

\begin{theorem}\label{t1}
A delta-matroid has a twist of width at most one if and only if it has no minor isomorphic to a member of $\mathcal{D}_{[5]}$. 
\end{theorem}
The proof of this theorem appears at the end of Section~\ref{s3}.

We note that the excluded minors of twists of matroids (i.e., twists of width zero delta-matroids) has been shown, but not explicitly stated, to be 
$( \{a\}, \{ \emptyset , \{a\} \} ) $, $D_3$, and $D_3\ast \{a\}$ by A. Duchamp in~\cite{adfund}. This result can be recovered from Theorem~\ref{t1} by restricting to even delta-matroids, where an \emph{even} delta-matroid is a delta-matroid in which the difference in size between any two feasible sets is even.

Above we  mentioned the close connection between delta-matroids and graphs in surfaces. The width of a delta-matroid can be viewed as the analogue of the genus (or more precisely the Euler genus) of an embedded graph, while twisting is the analogue of S.~Chmutov's partial duality of \cite{Ch09}. Thus characterising twists of width one is the analogue of characterising partial duals of graphs in the real projective plane. The topological graph theoretical analogues of Theorems~\ref{tm1} and~\ref{t1} can be found in \cite{mo13,mo16}.

\section{The proof of the Rough Structure Theorem}\label{s2}

For the convenience of the reader, we recall some standard matroid and delta-matroid terminology.  
Given a delta-matroid $D=(E,\mathcal{F})$ and element  $e\in E$, if $e$ is in every feasible set of $D$ then we say that $e$ is a \emph{coloop} of $D$.
If $e$ is in no feasible set of $D$, then we say that $e$ is a \emph{loop} of $D$.
If $e\in E$ is not a coloop, then $D$ \emph{delete} $e$, denoted by $D\ba e$, is the delta-matroid $(E-e, \{F : F\in \mathcal{F}\text{ and } F\subseteq E-e\})$.
If $e\in E$ is not a loop, then $D$ \emph{contract} $e$, denoted by $D/e$, is the delta-matroid $(E-e, \{F-e : F\in \mathcal{F}\text{ and } e\in F\})$.
If $e\in E$ is a loop or coloop, then $D/e=D\ba e$.
Useful identities that we use frequently are $ D/e = (D\ast  e) \ba e $ and $ D\ba e = (D\ast  e) / e $.
If $D'$ is a delta-matroid obtained from $D$ by a sequence of deletions and contractions, then $D'$ is independent of the order of the  deletions and contractions used in its construction, so we can define $D\ba X/Y$ for disjoint subsets $X$ and $Y$ of $E$, as the result of deleting each element in $X$ and contracting each element in $Y$ in some order.
A \emph{minor} of $D$ is any delta-matroid that is obtained from it by deleting or contracting some of its elements.
The \emph{restriction} of $D$ to a subset $A$ of $E$, written $D|A$, is equal to $D\ba \BA$. Note that if $\emptyset \in \F(D)$ then $F$ is feasible in $D|A$ if and only if $F\subseteq A$ and $F\in \F(D)$.

\medskip

The \emph{connectivity function} $\lambda_M$ of a matroid $M$ on ground set $E$ with rank function $r$  is defined on all subsets $A$ of $E$ by  $\lambda_M(A) = r(A) + r(\BA) - r(E)$. %from Oxley p272
Recall that  $A$ is said to be a \emph{separator} of $M$ if $A$ is a union of components of $M$. This happens if and only if $\lambda_M(A)=0$. Moreover, $A$ is a separator if and only if $\BA$ is a separator.
%from Oxley p128

We will use Bouchet's   analogue of the rank function for delta-matroids from~\cite{abrep}. For a delta-matroid $D=(E,\mathcal{F})$, it is denoted by $\rho_D$ or simply $\rho$ when $D$ is clear from the context. Its value on a subset $A$ of $E$ is given by \[\rho(A):=|E|-\min\{|A\bigtriangleup F| : F\in \mathcal{F}\}.\]

The following theorem determines the width of a twist of a delta-matroid.
\begin{theorem}\label{t2}
Let $D=(E,\F)$ be a delta-matroid and $A\subseteq E$. Then the width, $w(D*A)$, of the twist of $D$ by $A$ is given by
\[w(D*A)=w(D|A)+w(D|\BA) +2\,\lambda _{D_{\min}}(A).\]
\end{theorem}

\begin{proof}
The largest feasible set in $D*A$ has size $\max \{|F\btu A|:F\in \F (D)\}$.  
Take $F'\in\F$ such that $|F'\btu A|$ is maximal.  
Then $|F'\btu \BA |$ is minimal.  
As $\rho (\BA)=|E|-\min\{|F\btu \BA |:F\in\F\}$, we see that $\rho (\BA )=|E|-|F'\btu \BA |=|F'\btu A|$.  
Hence the largest feasible set in $D*A$ has size equal to $\rho (\BA )$.

Next, the size of the smallest feasible set in $D*A$ is $|E|$ minus the size of the largest feasible set in $(D*A)^*=D*\BA $. By an application of the above, it follows that the size of  the smallest feasible set in $D*A$ is $|E|-\rho (A)$.  
Hence $w(D*A)=\rho (\BA )-|E|+\rho (A)$.

We let $r$ and $n$ be the rank and nullity functions, respectively, of $D_{\min}$.  
From \cite{CMNR1}, we know that $w(D|A)=\rho (A)-r(A)-n(E)+n(A)$.  
As $n(A)=|A|-r(A)$ and $n(E)=|E|-r(E)$, 
\begin{align*}
w(D|A)&+w(D|\BA ) \\
&=\rho (A)-r(A)-|E|+r(E)+|A|-r(A) + \rho (\BA )-r(\BA )-|E|+r(E)+|\BA |-r(\BA )\\
&= \rho (\BA )-|E|+\rho (A) -2(r(A)+r(\BA )-r(E))
\\& =w(D*A)-2 (\lambda _{D_{\min}}(A)),
\end{align*}
giving the result. 
\end{proof}

The following two theorems are immediate consequences of Theorem~\ref{t2}. The Rough Structure Theorem, Theorem~\ref{tm1}, follows immediately from the second of them.

\begin{theorem}[Chun et al \cite{CMNR2}]\label{tt2}
Let $D=(E,\F)$ be a delta-matroid, $A\subseteq E$, and $\BA=E-A$. Then $D*A$ is a matroid if and only if  $A$ is a separator of $D_{\min}$,  and both $D|A$ and $D|\BA$ are matroids.
\end{theorem}

\begin{theorem}\label{tt}
Let $D=(E,\F)$ be a delta-matroid, $A\subseteq E$, and $\BA=E-A$. Then $D*A$ has width one if and only if  $A$ is a separator of $D_{\min}$, and one of $D|A$ and $D|\BA$ is a matroid and the other has width one.
\end{theorem}

For convenience, we write down the following straightforward corollary. It provides the form of the Rough Structure Theorem that we use to find excluded minors in the next section.
\begin{corollary}\label{c2}
Let  $D=(E,\F)$ be a delta-matroid in which $\emptyset$ is feasible. Then the following hold.
\begin{enumerate}
\item  $D$ has a twist of width zero  if and only if there exists $A\subseteq E$ such that $D|A$ and $D|\BA$ are both of width zero.
\item  $D$ has a twist of width one if and only if there exists $A\subseteq E$ such that 
 $D|A$ is a matroid, and  $D|\BA$ is of width one.
\end{enumerate}
\end{corollary}
\begin{proof}
This is a straightforward consequence of the fact that if $\emptyset$ is feasible in $D$, then $D_{\min}$ is the matroid on $E(D)$ where each element is a loop, thus every set $A\subseteq E$ is a separator of $D_{\min}$.
\end{proof}

\section{The proof of the excluded minor characterisation}\label{s3}

We begin this section by verifying that the class of delta-matroids in question is indeed minor-closed.

\begin{proposition}\label{p1}
For each $k\in \mathbb{N}_0$, the set of delta-matroids with a twist of width at most $k$ is minor-closed.
\end{proposition}
\begin{proof}
Let $D=(E,\mathcal{F})$ and suppose $w(D\ast A) \leq k$ for some $A\subseteq E$. If $E$ is empty the result is trivial, so assume not and let $e\in E$. 
If $e\notin A$ then
$(D\ba e)\ast A = (D\ast A)\ba e$, and
$(D/e)*A=((D*e)\ba e)*A= ((D*e)*A)\ba e = ((D*A)\ast  e) \ba e = (D*A) / e $. 
 Similarly, if $e\in A$ then $e\notin A-e$, so using and extending the previous argument,
$(D/e)\ast (A-e)= (D*(A-e))/e = ((D*A)* e)/e = (D*A)\ba e$, and 
$(D\ba e) \ast (A-e) = (D\ast (A-e)) \ba e  =  ((D\ast A)\ast e)  \ba e     = (D\ast A)/e $. 
In each case we see that  $D/e$  and $D\ba e$ have a twist that can be written as $(D\ast A)/e$ or  $(D\ast A) \ba e$.  Since deletion and contraction never increase width it follows that  $D/e$  and $D\ba e$ have twists of width at most $w(D\ast A) \leq k$. The result follows.
\end{proof}

\begin{lemma}\label{l1}
Let  $D=(E,\F)$ be a delta-matroid  and $A\subseteq E$. Then 
\[\{ H : H \text{ is a minor of } D\ast A \}  =  \{ J\ast (A\cap E(J)) : J \text{ is a minor of } D \}  .  \]
\end{lemma}
\begin{proof}
In the proof of Proposition~\ref{p1} it was shown that 
if $e\notin A$  then $(D*A) / e  =  (D/e)*A$  
and $ (D\ast A)\ba e = (D\ba e)\ast A $, whereas if $e\in A$ then  $(D*A)\ba e = (D/e)\ast (A-e)$ and 
$(D\ast A)/e = (D\ba e) \ast (A-e)$. The result follows immediately from this. 
\end{proof}

\begin{lemma}\label{l2}
Let  $D$ be a delta-matroid in which the empty set  is feasible. Then $D$ has a twist of width at most 1, or contains a minor isomorphic to one of  $D_1, \ldots,D_5$.
\end{lemma}

\begin{proof}
For any delta-matroid  $D$ in which the empty set is feasible, set
\[  L:= \{x\in E(D) : \{x\} \in \F(D)\}  \quad \text{and} \quad \BL=E(D)-L.\]
(Technically we should record the fact that $L$ depends upon $D$ in the notation, however we avoid doing this for notational simplicity. This should cause no confusion.) Note that $L$ may be empty. 
Construct a (simple) graph $G_D$ as follows. 
Take one vertex $v_x$ for each element $x\in  \BL$, and add one other vertex $v_L$. 
The edges of  $G_D$ arise from certain two-element feasible sets of $D$. 
Add an edge $v_xv_y$ to $G_D$ for each pair $x,y\in \BL$ with  $\{x,y\}\in \F(D)$; add an edge $v_xv_L$ to $G_D$ if $\{x,z\}\in \F(D)$ for some $z\in L$. 

We consider two cases: when $G_D$ is bipartite, and when it is not. We will show that if $G_D$ is bipartite then  $D$ must have a twist of width at most one or a minor isomorphic to $D_1$ or $D_2$; if $G_D$ is not bipartite then it must have a minor isomorphic to $D_1$, $D_3$, $D_4$, or $D_5$.

\medskip
\noindent\textbf{\textit{Case 1.}} Let  $D$ be a delta-matroid in which the empty set is feasible, and such that $G_D$ is bipartite. Fix a 2-colouring of $G_D$. Let 
$A$ be the set of elements in $E(D)$ that correspond to the vertices in the colour class containing $v_L$ together with the elements in $L$,
and let  $\BA\subseteq E(D)$ be the set of elements  corresponding to the vertices in the colour class not containing $v_L$. 

We start by showing
\begin{equation}\label{e1}
D|\BA \cong U_{0, |\BA|},
\end{equation}
where $U_{0, |\BA|}$ denotes the uniform matroid with rank zero and $|\BA|$ elements.

To see why \eqref{e1} holds, note that $\F(D|\BA)=\{F : F\subseteq \BA\text{ and } F\in \F(D)\}$.  
Since the elements in $\BA$ correspond to vertices in $\BL$,  
no feasible sets of $D|\BA$ have size one.
Furthermore,  $\F(D|\BA)$ cannot contain any sets of size two since, by the construction of $G_D$, whenever $\{x,y\}\in \F(D)$ the corresponding vertices $v_x$ and $v_y$ are in different colour classes. 
Since $\emptyset \in \F(D|\BA)$, the Symmetric Exchange Axiom ensures that there are no other feasible sets.   
(If $F\in \F(D|\BA)$ with $F\neq \emptyset$, take $x\in \emptyset \btu F$. Then by the Symmetric Exchange Axiom $\emptyset \btu \{x, y\}$ must be in $\F(D|\BA)$ for some $y$, but there are no feasible sets of size one or two.)
This completes the justification of \eqref{e1}. 
 
 Next we examine the feasible sets in $D|A$. Trivially $\emptyset\in \F(D|A)$. 
 The set of feasible sets of $D|A$ of size one is 
$\{ F\in  \F(D|A) :  |F|= 1\} =\{ F\in  \F(D) :  |F|= 1\}  = \{ \{x\} : x\in L \}$.

If   $\F(D|A)$ contains a set $\{x,y\}$ of size two then $x,y\in L$ as otherwise there would be an edge $v_xv_y$ in $G_D$ whose ends are in the same colour class.  It follows in this case that $D|A$ and hence $D$ contains a minor isomorphic to $D_1$.

Now assume that $\F(D|A)$ does not contain a set of size two. 
If $\F(D|A)$ has no sets of size one then, arguing via the Symmetric Exchange Axiom as in the justification of \eqref{e1}, we  have $D|A \cong U_{0,|A|}$. Taken together with~\eqref{e1}, this implies that $A$ satisfies the conditions of the first part of 
Corollary~\ref{c2}, so $D$ has a twist of width zero.

Suppose that $\F(D|A)$ does contain a set of size one.  If it contains no sets of size greater than one then  $D|A$ is of width one, and by combining this with~\eqref{e1}, it  follows from Corollary~\ref{c2} that $D$ has a twist of width one  ($D\ast A$ and $D\ast \BA$ are such twists).
On the other hand,  if  $\F(D|A)$ does contain a set of size greater than one, then, as it does not contain   
a set of size two, the Symmetric Exchange Axiom guarantees there is a set  in  $\F(D|A)$ of size exactly three. (If not, let $F$ be a minimum sized feasible set with $|F|>3$.  Then $F \ba \{x,y\}$ is feasible and of size at least two for some $x,y\in \emptyset \btu F$  contradicting the minimality of $|F|>3$.) Let $\{x,y,z\}\in \F(D|A)$. Then after possibly relabelling its elements, the collection of feasible sets of $D|\{x,y,z\}$ is one of 
\[
  \{ \emptyset , \{x\}, \{y\}, \{z\}, \{x,y,z\}  \}, \quad
     \{ \emptyset , \{x\}, \{y\},  \{x,y,z\}  \}, \quad
     \{ \emptyset , \{x\} , \{x,y,z\}  \}.\]
Only the first of the three cases is possible as 
the Symmetric Exchange Axiom fails for the other two showing that neither is the collection of feasible sets of a delta-matroid.
Hence, restricting $D$ to $\{x,y,z\}$ results in a minor isomorphic to $D_2$. 

Thus we have shown that if $G_D$ is bipartite then $D$ has a twist of width at most one or contains a minor isomorphic to $D_1$ or $D_2$. This completes the proof of Case 1.

\medskip
\noindent\textbf{\textit{Case 2.}} Let  $D$ be a delta-matroid in which the empty set is feasible, and such that $G_D$ is non-bipartite. We will show that $D$ contains a minor isomorphic to one of $D_1$, $D_3$, $D_4$ or $D_5$ by induction on the length of a shortest odd cycle in $G_D$. 

For the base of the induction suppose that $G_D$ has an odd cycle $C$ of length three. There are two sub-cases, when $v_L$ is not in $C$ and when it is. Note that the former sub-case includes the situation where $L=\emptyset$.

\smallskip
\noindent\textbf{\textit{Sub-case 2.1.}} Suppose that $v_L$ is not in $C$. Let $x,y,z\in E(D)$ be the elements corresponding to the three vertices  of $C$. We have $x,y,z\in \BL$, so $\{x\}, \{y\}, \{z\}\notin \F(D)$.  From the three edges of $C$ we have $\{x,y\}, \{y,z\}, \{z,x\}\in \F(D)$. It follows that $D|\{x,y,z\}$ is isomorphic to either $D_3$ or $D_4$ giving the required minor.

\smallskip
\noindent\textbf{\textit{Sub-case 2.2.}} Suppose that $v_L$ is in $C$. Let $v_x, v_y, v_L$ be the vertices in $C$. The edges of $C$ give that $ \{x,y\} \in \F(D)$,  and since $x,y\in \BL$ we have $\{x\}, \{y\} \notin \F(D)$. We also know that there are elements $\al, \be\in L$ such that $\{\al\}, \{\be\}, \{x,\al\}  ,\{y,\be\} \in  \F(D)$, where possibly $\al=\be$.

If $\al=\be$ then $D|\{x,y,\al\}$ must have feasible sets 
\begin{equation}\label{e13}
  \{\emptyset , \{\al\} , \{x,\al\}  ,\{y,\al\} ,\{x,y\}\}  \quad\text{or}\quad   \{\emptyset , \{\al\} , \{x,\al\}  ,\{y,\al\} ,\{x,y\}, \{\al, x,y\}\}.    \end{equation}
The first case gives a minor of $D$ isomorphic to $D_5$; in the second case, $(D|\{x,y,\al\}) / \al$ is a  minor of $D$ isomorphic to $D_1$.

If $\al\neq \be$ then the feasible sets  of $D|\{x,y,\al,\be\}$ of size zero or one are exactly  $\emptyset$, $\{\al\}$, and $\{\be\}$. 
From $G_D$, the feasible sets of size two include  $\{x,\al\}  ,\{y,\be\} ,  \{x,y\}$. 
If $\{ y,\al\}$ is also feasible then $D|\{x,y,\al\}$ is isomorphic to one of the delta-matroids arising from \eqref{e13}, so $D$ has a minor isomorphic to $D_1$ or $D_5$.
The case when $\{ x,\be\}$ is feasible is similar.
If $\{\al,\be\}$ is feasible then  $D|\{\al,\be\}$ is isomorphic to $D_1$.

The case that remains is when the feasible sets of $D|\{x,y,\al,\be\}$ of size at most two are exactly
\[  \emptyset , \{\al\}, \{\be\}, \{x,\al\}  ,\{y,\be\} ,  \{x,y\}. \]
By applying the Symmetric Exchange Axiom to each of the pairs of feasible sets
$(\{\al\},\{x,y\})$,  $(\{\be\},\{x,y\})$, $(\{\be\},\{x,\al\})$ and $(\{\al\},\{y,\be\})$, one can show that each of the three element sets, \[ \{\al,x,y\},   \{\be,x,y\}, \{\al,\be,x\},  \{\al,\be,y\}, \]
is feasible in $D|\{x,y,\al,\be\}$. Finally, $\{\al,\be, x,y\}$ may or may not be feasible. 

If $\{\al,\be, x,y\}$ is feasible then $(D|\{x,y,\al,\be\})/\{x,y\}$ is isomorphic to $D_1$; if $\{\al,\be, x,y\}$ is not feasible then $(D|\{x,y,\al,\be\})/\{\al\}$ is isomorphic to $D_5$.

This completes the base of the induction.

For the inductive hypothesis, we assume that, for some $n>3$, if $D$ is a delta-matroid such that $\emptyset\in\mathcal{F}(D)$ and $G_D$ has an odd cycle of length less than $n$, then $D$ has a minor isomorphic to $D_1,D_3,D_4$, or $D_5$.

Suppose that $\emptyset\in\mathcal{F}(D)$ and a shortest odd cycle $C$ of
$G_D$ has length $n$. Again there are two sub-cases: when $v_L$ is not in $C$ and when it is.

\smallskip
\noindent\textbf{\textit{Sub-case 2.3.}} Suppose that $v_L$ is not in $C$. Let $C=v_{x_1}v_{x_2} \ldots v_{x_n}v_{x_1}$. Since each $x_i\in \BL$ and $C$ is the shortest odd cycle in $G_D$, 
\begin{equation}\label{e5}
 \emptyset , \{ x_1,x_2 \},   \{ x_2,x_3 \}, \ldots,\{ x_n,x_1 \}   
\end{equation}
is a complete list of the feasible sets of size at most two in $D|\{x_1,\ldots, x_n\}$. 

Next, we show
\begin{equation}\label{e3}
\{x_i,x_j,x_k\} \notin \F(D|\{x_1,\ldots, x_n\}), \quad \text{ for any distinct } 1\leq i,j,k\leq n.
\end{equation}
To see why \eqref{e3} holds, first note that, since $n>3$, every set of three distinct vertices in the cycle includes a non-adjacent pair.  If $\{x_i,x_j,x_k\}$ were feasible in $D|\{x_1,\ldots, x_n\}$, then, without loss of generality, $\{x_j,x_k\}\notin\mathcal{F}(D|\{x_1,\ldots, x_n\})$.  As $x_i\in\{x_i,x_j,x_k\}\btu\emptyset$, an application of the Symmetric Exchange Axiom would imply that $\{x_i,x_j,x_k\}\btu\{x_i,z\}$ is feasible for some $z\in\{x_i,x_j,x_k\}$.  Thus $\{x_j,x_k\},\{x_j\}$, or $\{x_k\}$ would be feasible, a contradiction to~\eqref{e5}.  Thus~\eqref{e3} holds.

Next we show that, taking indices modulo $n$,
\begin{equation}\label{e4}
\{x_i,x_{i+1},x_j,x_{j+1}\} \in \F(D|\{x_1,\ldots, x_n\}),
\end{equation}
for any $i$ and $j$ such that $1 \leq i,j \leq n$ and $i,i+1,j,j+1$ are pairwise distinct.

For this,
first suppose that neither $x_{i+1}$ and $x_j$ nor $x_{j+1}$ and $x_i$ are adjacent in $C$. 
Then by \eqref{e5}, $\{ x_i,x_{i+1}\}$ and  $\{ x_j,x_{j+1}\}$ are feasible.  As $x_{j}$ is in their symmetric difference, by the Symmetric Exchange Axiom, $\{x_i,x_{i+1}\}\btu\{x_j,y\}$ is feasible for some $y\in\{x_i,x_{i+1},x_j,x_{j+1}\}$.  Thus $\{x_i,x_{i+1},x_j\}, \{x_i,x_j\}, \{x_{i+1},x_j\}$ or $\{x_i,x_{i+1},x_j,x_{j+1}\}$ is feasible.  By~\eqref{e5} and \eqref{e3}, $\{x_i,x_{i+1},x_j,x_{j+1}\}$ is feasible.
If $x_{i+1}$ and $x_j$ are adjacent then the Symmetric Exchange Axiom implies that $\{x_i,x_{i+1}\}\btu\{x_{i+3},z\}$ is feasible for some $z\in\{x_i,x_{i+1},x_{i+2},x_{i+3}\}$.  Again,~\eqref{e5} and~\eqref{e3} imply that $\{x_i,x_{i+1},x_{i+2},x_{i+3}\}$ must be feasible. The other case is identical.  
This completes the justification of \eqref{e4}.

Combining \eqref{e5}--\eqref{e4} gives that all of 
$ \emptyset , \{ x_1,x_2 \},   \{ x_2,x_3 \}, \ldots,\{ x_{n-2},x_1 \} $, 
but none of $\{x_1\}, \ldots, \{x_{n-2}\}$, 
 are  feasible in $(D|\{x_1,\ldots, x_n\})/\{ x_{n-1},x_n\}$. Hence the graph $G_{(D|\{x_1,\ldots, x_n\})/\{ x_{n-1},x_n\}}$  has a shorter odd cycle than $G_D$. By the inductive hypothesis, $(D|\{x_1,\ldots, x_n\})/\{ x_{n-1},x_n\}$ and hence $D$ has a minor isomorphic to one of $D_1$, $D_3$, $D_4$ or $D_5$.

\smallskip
\noindent\textbf{\textit{Sub-case 2.4.}} Suppose that $v_L$ is in $C$. Let $C=v_Lv_{x_2}v_{x_3}\dots v_{x_n}v_L$.
The edges of the cycle give that, for each $2\leq i\leq n-1$,  $ \{x_i,x_{i+1}\} \in \F(D)$.  Also, for $2\leq i\leq n$, since $x_i\in \BL$ we have $\{x_i\} \notin \F(D)$. 
We also know that there are elements $\al, \be\in L$ such that $\{\al\}, \{\be\}, \{\al, x_2\}  ,\{\be,x_n\} \in  \F(D)$ where possibly $\al=\be$. (This possibility is covered in the following analysis.)

When $\al\neq\be$, if $\{\al,\be\}\in\mathcal{F}(D)$, then $D|\{\al,\be\}$ is isomorphic to $D_1$,  therefore we assume $\{\al,\be\}\notin\mathcal{F}(D)$. Using that $C$ is a shortest odd cycle, the feasible sets of $D|\{\al, \be, x_2,\ldots, x_n\}$ of size at most two are exactly
\begin{equation}\label{e6}
 \emptyset , \{\al\} , \{\be\}, \{\al,x_2\},    \{ x_2,x_3 \},   \{ x_3,x_4 \}, \ldots,  \{ x_{n-1},x_n \}, \{ \be, x_n\}    . 
\end{equation}
An argument similar to the justification of \eqref{e3} gives that 
\begin{equation}\label{e7}
\{x_i,x_j,x_k\} \notin \F(D|\{\al,\be, x_2,\ldots, x_n\}), \quad \text{ for any distinct } 2\leq i,j,k\leq n.
\end{equation}
However 
\begin{equation}\label{e8}
\{\al, x_{n-1}, x_n\},  \{\be, x_{n-1}, x_n\} \in \F(D|\{\al,\be,x_2,\ldots, x_n\}).
\end{equation}
To see this note that $x_{n-1} \in \{\al \} \btu \{ x_{n-1},x_n \}$, so the Symmetric Exchange Axiom gives that one of $\{\al,  x_{n-1} \}$, $\{x_{n-1}\}$, or $\{\al,  x_{n-1}, x_n \}$ is feasible, and  we know from \eqref{e6} that the feasible set must be the third option.   
That $\{\be, x_{n-1}, x_n\}$ is feasible follows from a similar argument.

We next show that for each $2\leq i < n-2$,
\begin{equation}\label{e9}
\{\al, x_2 , x_{n-1}, x_n\}  , \{x_i, x_{i+1} , x_{n-1}, x_n\} , \{\be, x_{n-2} , x_{n-1}, x_n\} \in \F(D|\{\al,\be,x_2,\ldots, x_n\}).
\end{equation}
For this, first consider $x_2 \in \{ x_{n-1}, x_n\}\btu \{\al, x_2\}$. The Symmetric Exchange Axiom implies that $\{x_{n-1},x_n\}\btu\{x_2,z\}$ is feasible for some $z\in\{\al,x_2,x_{n-1},x_n\}$.  By~\eqref{e6} and \eqref{e7}, $z=\al$, thus $\{\al,x_2,x_{n-1},x_n\}$ is feasible. 
Next, to show that $\{x_i,x_{i+1},x_{n-1},x_n\}$ is feasible, we take $x_i\in\{x_{n-1},x_n\} \btu\{x_i,x_{i+1}\}$ and apply the Symmetric Exchange Axiom as above to see that $\{x_{n-1},x_n\}\btu\{x_i,z\}$ is feasible, where $z$ must equal $x_{i+1}$. 
Lastly, to show that $\{\be,x_{n-2},x_{n-1},x_n\}$ is feasible, we first show that $\{\be,x_{n-2},x_n\}\notin \F(D|\{\al,\be,x_2,\ldots, x_n\})$. If $\{\be,x_{n-2},x_n\}$ were feasible, then since $x_{n-2}\in \emptyset\btu\{\be,x_{n-2},x_n\}$, the Symmetric Exchange Axiom would give $\{x_{n-2}\}$, $\{\be,x_{n-2}\}$ or $\{x_{n-2},x_n\}$ as feasible, a contradiction. Now showing that $\{\be,x_{n-2},x_{n-1},x_n\}$ is feasible comes from taking $x_{n-2}\in\{\be,x_n\}\btu\{x_{n-2},x_{n-1}\}$. The Symmetric Exchange Axiom gives that $\{\be,x_n\}\btu\{x_{n-2},z\}$ is feasible for some $z\in\{\be,x_{n-2},x_{n-1},x_n\}$, of which $z=x_{n-1}$ is the only possibility.

From \eqref{e6}--\eqref{e9} it follows that all of $\emptyset$ ,   $\{\al\}$, $\{\be\}$, $\{\al,x_2\}$, $\{ x_2,x_3 \}$, $\ldots$, $\{ \be, x_{n-2} \}$, but none of $\{x_2\}$, $\ldots$, $\{x_{n-2}\}$,  are  feasible in $(D|\{\al, x_2,\ldots, x_n,\be\})/\{ x_{n-1},x_n\}$. Hence the graph $G_{(D|\{\al, x_2,\ldots, x_n,\be\})/\{ x_{n-1},x_n\}}$  has a shorter odd cycle than $G_D$. The inductive hypothesis gives that $(D|\{\al, x_2,\ldots, x_n,\be\})/\{ x_{n-1},x_n\}$ and hence $D$ has a   minor isomorphic to one of $D_1$, $D_3$, $D_4$ or $D_5$. 
This completes the proof of the sub-case, and the lemma.
\end{proof}

We now apply Lemma~\ref{l2} to prove our excluded minor characterisation of the family of delta-matroids admitting a twist of width at most one.

\begin{proof}[Proof of Theorem~\ref{t1}]
All twists of the delta-matroids $D_1, \ldots,D_5$ are of width at least two. Since the set of delta-matroids with a twist of width at most one is minor-closed it follows that 
no minor of a delta-matroid with a twist of width at most one is isomorphic to 
a member of $\mathcal{D}_{[5]}$. This proves one direction of the theorem.

Conversely suppose that every twist of a delta-matroid $D=(E,\F)$ is of width at least two. Let $A\in \F$. Then $D*A$ is a delta-matroid in which $\emptyset$ is feasible and in which every twist is of width at least two. By Lemma~\ref{l2}, $D*A$ has a minor isomorphic to one of $D_1, \ldots,D_5$. It follows from Lemma~\ref{l1} that $D$ has a minor isomorphic to a member of $\mathcal{D}_{[5]}$.
\end{proof}

\bibliographystyle{plain}
\bibliography{bib}

\end{document}